\documentclass[11pt,twoside]{amsart}
\usepackage{amsfonts}
\usepackage{amsmath}
\usepackage{amssymb}
\usepackage[latin1]{inputenc} % pour les accons
\usepackage{mathrsfs}
\newtheorem{theo}{Theorem}
\newtheorem{prop}{Proposition}

\newtheorem{lem}{Lemma}
\theoremstyle{definition}
\newtheorem{defn}{Definition}

\theoremstyle{remark}
\newtheorem{rem}{Remark}

\numberwithin{equation}{section}

\newcommand{\cal}       {\mathcal}
\def\finpr{\hfill \hbox{
\vrule height 1.453ex  width 0.093ex  depth 0ex \vrule height
1.5ex  width 1.3ex  depth -1.407ex\kern-0.1ex \vrule height
1.453ex  width 0.093ex  depth 0ex\kern-1.35ex \vrule height
0.093ex  width 1.3ex  depth 0ex}}

\def\bC{{\Bbb C}}
\def\bR{{\Bbb R}}

                   \def\l{\lambda}
       \def\ds{\displaystyle}  \let\w=\wedge
             
                  \let\l=\rightarrow
               
                 \def\ds{\displaystyle}
\let\ov=\overline
    
\def\1{\hbox{\rm1\kern-2.2pt{l}}\,}

\title[$J$-analytic subsets and Lelong numbers on almost complex manifolds]{Poincar\'e-Lelong formula, $J$-analytic subsets and Lelong numbers of currents on almost complex manifolds}
\author[F. Elkhadhra]{Fredj Elkhadhra}
\address{Département de Mathématique\\ Faculté des sciences de Monastir\\ 5000 Monastir Tunisie.}
\email{fredj.elkhadhra@fsm.rnu.tn}
\subjclass[2000]{Primary 32C30; Secondary 32Q60}
\keywords{Positive current, Almost complex manifold}
\begin{document}

\begin{abstract}
    In this paper, we first establish a Poincar\'e-Lelong type formula in the almost complex setting. Then, after introducing the notion of $J$-analytic subsets, we study the restriction of a closed positive current defined in an almost complex manifold $(M,J)$ on a $J$-analytic subset. Finally, we prove that the Lelong numbers of a plurisubharmonic current defined on an almost complex manifold are independent of the coordinate systems.
\end{abstract}
\maketitle
\section{Preliminaries and notations}
    Let $M$ be a real-dimentional manifold equipped with a smooth section $J\in{\rm End}(TM)$ for which $J^2=-\1$; we call $J$ the {\it almost complex structure} on $M$, and $(M,J)$ an {\it almost complex manifold}. Denote by $TM$ the tangent bundle to $M$ and by $TM^\star$ the associated dual bundle. Then, by extending $J$ to a $\bC$-linear automorphism of the complexification $T_\bC M$, it induces the splitting $$T_\bC M=T^{1,0}M\oplus T^{0,1}M ,$$ where $T^{1,0}M$ and $T^{0,1}M$ are the complex subbundles of $T_\bC M$ associated respectively with the $(+i)$ and $(-i)$-eigenspaces of $J$. Also we have $T_\bC ^*M={T^*_{1,0}}M\oplus {T^*_{0,1}}M$, where ${T^*_{1,0}}M$ and ${T^*_{0,1}}M$ are the dual bundles of $T^{0,1}M$ and $T^{0,1}M$ respectively. We also consider ${\mathscr C }_{p,q}^\infty(M,\bC)={\mathscr C }^\infty(\bigwedge^{p,q}T_\bC^*M,\bC)$ (resp.\ ${\mathscr D }_{p,q}(M,\bC)$) as the space of complex differentials (resp.\ differentials with compact support) of bidegree $(p,q)$ on $M$. The dual ${\mathscr D }'_{p,q}(M,\bC) $ is the space of currents of bidimension $(p,q)$ or of bidegree $(n-p,n-q)$. Recall that a current $T\in {\mathscr D }'_{p,q}(M,\bC)$ is nothing but a differential form of bidegree $(p,q)$ with distribution coefficients. Let $d$ be the De Rham exerior derivative, which splits as $d=\partial_J +\overline \partial_J -\theta_J  -\ov\theta_J\,$, where the operators $\partial_J $, $\overline \partial_J $, $\theta_J$ and $\ov\theta_J$ are of type respectively $(1,0)$, $ (0,1)$, $(2,-1)$ and $(-1,2)$. As is well-known the structure $J$ need not be integrable, i.e.\ $J$ need not be induced from local complex coordinates. According to the celebrated Newlander-Nirenberg theorem, this is the case if and only if the torsion tensor $\theta_J$ vanishes, or equivalently, if and only if $\overline\partial_J^2=0$. The following definition will be useful.
    \begin{defn}\
        \begin{enumerate}
            \item A form $\varphi\in{\mathscr C }_{p,p}^\infty(M,\bC)$ is said to be  positive if for all vectors $\xi_1,\cdots,\xi_p\in TM$ one has   $\varphi(\xi_1,J\xi_1,\cdots,\xi_p,J\xi_p)\geq 0$.  We say that $\varphi$ is   strongly positive if it can be written as $\varphi=\sum_{j=1}^N \lambda_ji\alpha_{1,j}\wedge\overline{\alpha}_{1,j} \wedge \cdots\wedge   i\alpha_{p,j}\wedge\overline{\alpha}_{p,j},$ where $\lambda_j\geq 0$ and $\alpha_{k,j}\in{\mathscr C }_{1,0}(M,\bC)$, for $j=1,\cdots,N$.
            \item Let $T$ be a current of bidimension $(p,p)$ on $(M,J)$. We say  that $T$ is positive if $T\w\varphi$ is a positive Radon   measure for every form $\varphi$ of bidegree $(p,p)$ on $M$ which is   strongly positive. The current $T$ is said to be closed if $dT=0$, and   plurisubharmonic if the current   $i\partial_J\overline\partial_JT$ is positive.
            \item Let $u$ be an upper semi-continuous function on $M$. One says that $u$ is $J$-plurisubharmonic ($J$-psh for short) if $u\circ \gamma$ is a subharmonic function for every $J$-holomorphic curve.
        \end{enumerate}
    \end{defn}
    Regarding positive currents, it is known that there are two important kinds of examples of positive currents \cite{7}. The first kind consists of currents of integration over almost complex submanifolds; they are closed currents. Recall that a real smooth submanifold $N$ of $M$ is said to be almost complex submanifold if the tangent bundle $TN$ is invariant by $J$. The second kind consists of currents of the form $i{\partial_J\overline\partial_J u}$, where $u$ is a $J$-psh function on $M$ such that $u\not\equiv -\infty\,$; they are not necessarily closed when $J$ is non integrable. Assume that $\omega\in{\mathscr C }_{1,1}^\infty(M,\bC)$ is a hermitian metric on $(M,J)$ and let $(\xi_1,\cdots,\xi_n )$ be an $\omega$-orthonormal local basis for the bundle $T^{1,0}M$. Hence, with respect to the above basis we have $\omega=(i/2)\sum_{j=1}^n\xi^*_j\wedge \overline{\xi}_j^*$. Let $\Omega$ be an open subset of $M$ and $A$ be a closed subset of~$\Omega$. A current $T$ of bidimension $(p,q)$ with zero order on $\Omega\smallsetminus A$ can be written as a differential form with measure coefficients, i.e.\ $T=\sum_{|I|=p,|J|=q}T_{IJ}\xi^*_I\wedge \overline{\xi}_J^*$ where $T_{IJ}$ are Radon measures on $\Omega\smallsetminus A$. We say that $\widetilde{T}$ exists if for all $I,J$ the measure $T_{IJ}$ is locally finite near the points of $A$. In this case $\widetilde T$ is the trivial extension of $T$ (i.e.\ $\widetilde T=0$ on $A$). When $p=q$ and $T$ is positive, the measure $T\w\omega^p$ is the mass measure of $T$ (see \cite{7} for more details). Consequently, $\widetilde{T}$ exists if and only if the mass measure $T\w\omega^p$ is locally finite in a neighborhood of any point of~$A$. The structure of the paper is as follows. In section 2, we give the analogue of the Poincar\'e-Lelong formula in the almost complex setting. In section 3, we introduce the notion of $J$-analytic subsets and we consider the restriction of a closed positive current on such a subset. In section 4, we complete the work begun by \cite{3} concerning the existence of the Lelong numbers of a plurisubharmonic current and we prove their independence from the almost complex coordinates.

\section{Poincar\'e-Lelong formula}
    Let $f$ be a holomorphic function on an open set $\Omega$ of $\bC$ and let us denote by $[Z]=\sum m_j[Z_j]$ the current of integration on the zero divisor of $f$, (here $Z_j$ are the irreducible components of $Z$ and $m_j$ is the multiplicity of $f$ on $Z_j$). The classical Poincar\'e-Lelong formula states that ${i\over \pi}\partial\ov\partial\log |f|=[Z]$ in the sense of currents. Roughly speaking, if $\Delta$ is the Laplacien operator, then the trace ${1\over {2\pi}}\Delta\log |f|$ is the euclidean area measure of $Z$. We refer the reader to \cite{1} and \cite{6} for a more general formula involving holomorphic sections of a hermitian bundle. Our aim in this section is to establish an analogue of the Poincar\'e-Lelong-King formula in the almost complex category. Let $\Omega$ be an open set of $\bR^{2n}$ equipped with an almost complex structure $J$ of class~${\mathscr C }^3$. Given a ${\mathscr C }^3$-submanifold $Z$ of $\Omega$ of codimension $2p$ provided $J(TZ)=TZ$. Let $U$ be an open subset of $\Omega$ such that $Z$ is defined on $U$ by $f_1=\cdots=f_p=0$, where $f_j$ are of class ${\mathscr C }^3$ on $U$, $\ov\partial_Jf_j=0$ on $Z\cap U$ and $\partial_Jf_1\wedge\cdots\wedge\partial_J f_p\not=0$ on $U$ (see \cite[prop.1]{2}). With these notations, our generalization of the Poincar\'e-Lelong-King formula can be stated as follows :
    \begin{theo}\label{th1}
        Denote $f=(f_j)_{1\le j \le p}$, $|f|^2=\ds\sum_{j=1}^p|f_j|^2$. Then we have :
        $$\left({i\over \pi}\partial_J\ov\partial_J\log |f|^2\right)^p=[Z]+ R_J(f),$$
        where $R_J(f)$ is a $(p,p)$-current which has $L_{loc}^{\alpha}$ integrable as coefficients, where $\alpha<1+{1\over{2p-1}}$. Moreover, $R_J(f)=0$ when the structure $J$ is integrable.
    \end{theo}
Observe that the last formula gives the expression of the current of integration associated to $Z$ in terms of $f=(f_1,\cdots,f_p)$. Moreover, Theorem \ref{th1} guarantees that the Monge-Amp\`ere current $\left({i\over \pi}\partial_J\ov\partial_J\log |f|^2\right)^p$ is well defined on $U$ and the extra current $R_J(f)=\left({i\over \pi}\partial_J\ov\partial_J\log |f|^2\right)^p_{|U\smallsetminus Z}$ appears as the obstruction of the closedness  and the positivity of the above Monge-Amp\`ere current.
\begin{proof}
For $\varepsilon>0$, we have
\begin{equation}\label{eq2.1}
    \left(i\partial_J\ov\partial_J\log(|f|^2+\varepsilon^2)\right)^p=\frac{(i\partial_J\overline\partial_J|f|^2)^p}{(|f|^2+\varepsilon^2)^p} -p\frac{(i\partial_J\overline\partial_J|f|^2)^{p-1}\wedge i\partial_J|f|^2\wedge\overline\partial_J|f|^2}{(|f|^2+\varepsilon^2)^{p+1}}=w_1(f,\varepsilon)-w_2(f,\varepsilon).
\end{equation}
Our starting point in the study of $w_1(f,\varepsilon)$ is the following direct computation :
$$
\begin{array}{lcl}
\left(i\partial_J\overline\partial_J|f|^2\right)^p&=&\left(i\ds\sum_{j=1}^p \partial_J\overline\partial_J|f_j|^2\right)^p\\&=&\left[i\ds\sum_{j=1}^p \partial_J f_j\wedge\overline{\partial_J f_j}+i\ds\sum_{j=1}^p f_j\partial_J \overline{\partial_J f_j}+i\ds\sum_{j=1}^p \partial_J \overline f_j\wedge\overline\partial_J f_j+i\ds\sum_{j=1}^p\overline f_j\partial_J \overline\partial_J f_j\right]^p\\&=&\left[i\ds\sum_{j=1}^p \partial_J f_j\wedge\overline{\partial_J f_j}+i\ds\sum_{j=1}^p \partial_J \overline f_j\wedge\overline\partial_J f_j+2{\cal R}e\left(i\ds\sum_{j=1}^p \overline f_j\partial_J \overline\partial_J f_j\right)\right]^p.
\end{array}
$$
Since $\ov\partial_Jf_j=0$ on $Z$ for $j=1,\cdots,p$, it follows that there exist $(0,1)$-forms $\alpha_k^j,\beta_k^j$ of class ${\cal C}^1$ such that $\ov\partial_Jf_j=\ds\sum_{k=1}^p (f_k\alpha_k^j+\ov f_k\beta_k^j)$. Therefore, we get the following estimates :
    $$
    \begin{array}{lcl}
        i\ds\sum_{j=1}^p \partial_J \overline f_j\wedge\overline\partial_J f_j&=&\ds\sum_{j=1}^p i\left(\ds\sum_{k=1}^p \ov f_k\ov\alpha_k^j+f_k\ov\beta_k^j\right)\wedge\left(\ds\sum_{s=1}^p f_s\alpha_s^j+\ov f_s\beta_s^j\right)\\
        &=&\ds\sum_{1\leq j,k,s\leq p}\left(\ov f_kf_s \overline \alpha_k^j\wedge\alpha_k^s+\overline f_k\ov f_s\ov\alpha_k^j\wedge\beta_s^j+f_k f_s\overline\beta_k^j\wedge\alpha_s^j+f_k \ov f_k\overline\beta_k^j\wedge\beta_s^j\right)\\
\noalign{\vskip4pt}
        &=&{\cal O}(|f|^2) ,
\end{array}$$
where ${\cal O}\left(|f|^2\right)$ is a current which has ${\cal O}\left(|f|^2\right)$ as coefficients. Also, we have
$$
\begin{array}{lcl}
i\ds\sum_{j=1}^p\overline f_j\partial_J \overline\partial_J f_j &=& \ds\sum_{1\leq j,k\leq p} \left(i\overline f_j\partial_J f_k\wedge\alpha_k^j+\ov f_jf_k\partial_J\alpha_k^j+ \overline f_j\partial_J\overline f_k\wedge\beta_k^j+\overline f_j\ov f_k\partial_J\beta_k^j\right)\\&=&\ds\sum_{1\leq j,k\leq p} i\left(i\overline f_j\partial_J f_k\wedge\alpha_k^j+\ov f_jf_k\partial_J\alpha_k^j\right)+\ds\sum_{1\leq j,k,s\leq p}i\left(\overline f_j\ov f_s\ov\alpha_s^k\wedge\beta_k^j+\ov f_j f_s\ov\beta_s^k\wedge\beta_k^j\right)\\&=& \ds\sum_{1\leq j,k\leq p} i\overline f_j\partial_J f_k\wedge\alpha_k^j+{\cal O}\left(|f|^2\right)={\cal O}\left(|f|\right).
\end{array}
$$
By adding the last two equations involved in the expression of $w_1(f,\varepsilon)$ we obtain
\begin{equation}\label{eq2.2}
w_1(f,\varepsilon)={{\bigg(i\ds\sum_{j=1}^p \partial_J f_j\wedge\overline{\partial_J f_j}+{\cal O}(|f|)\bigg)^p}\over{(|f|^2+\varepsilon^2)^p}}=
p!{{i\partial_J f_1\wedge\overline{\partial_J f_1}\wedge\cdots\wedge i\partial_J f_p\wedge\overline{\partial_J f_p}}\over{(|f|^2+\varepsilon^2)^p}}+{{{\cal O}\left(|f|\right)}\over{(|f|^2+\varepsilon^2)^p}},
\end{equation}
Let us now compute the second form $w_2(f,\varepsilon)$ of (\ref{eq2.1}). For this, we begin with :

$$
\begin{array}{lcl}
i\partial_J|f|^2\wedge\overline\partial_J|f|^2&=&i\left[\ds\sum_{j=1}^p f_j\partial_J\overline f_j+\overline f_j\partial f_j\right]\wedge\left[\ds\sum_{k=1}^p f_k\overline{\partial_J f_k}+f_k\overline \partial f_k\right]\\&=&2{\cal R}e\left(\ds\sum_{1\leq j,k\leq p} f_j f_k\partial_J \ov f_j\wedge\overline{\partial_J f_k}\right)+\ds\sum_{1\leq j,k,r,s\leq p} if_j f_k\left(\ov f_s\ov\alpha_s^j+f_s\ov\beta_s^j\right)\wedge\left(f_r\alpha_r^k+\ov f_r\beta_r^k\right)\\& &+\ds\sum_{1\leq j,k\leq p} \overline f_j f_k\partial_J f_j\wedge\overline{\partial_J f_k}\\&=&i\ds\sum_{1\leq j,k\leq p} \overline f_j f_k\partial_J f_j\wedge\overline{\partial_J f_k}+{\cal O}\left(|f|^3\right).
\end{array}
$$
It is immediate to get ${\cal O}\left(|f|^3\right)$ because the first term in the second equality of the preceding equation is a ${\cal O}\left(|f|^3\right)$, while the second one has ${\cal O}\left(|f|^4\right)$ as coefficients. By means of the above estimation and turning back to the expression of $w_2(f,\varepsilon)$, we derive that :
$$
\begin{array}{lcl}
(|f|^2+\varepsilon^2)^{p+1}w_2(f,\varepsilon)&=&p\left(i\partial_J\overline\partial_J|f|^2\right)^{p-1}\wedge i\partial_J|f|^2\wedge\overline\partial_J|f|^2\\&=&
p\left[\ds\sum_{j=1}^p i\partial_J f_j\wedge\overline{\partial_J f_j}+{\cal O}\left(|f|\right)\right]^{p-1}\wedge\left[\ds\sum_{1\leq j,k\leq p} \overline f_j f_ki\partial_J f_j\wedge\overline{\partial_J f_k}+{\cal O}\left(|f|^3\right)\right]
\\&=&p\left[\left(\ds\sum_{j=1}^p i\partial_J f_j\wedge\overline{\partial_J f_j}\right)^{p-1}+{\cal O}\left(|f|\right)\right]\wedge\left[i\gamma\wedge\ov\gamma+{\cal O}\left(|f|^3\right)\right],
\end{array}
$$
where $\gamma=\ds\sum_{j=1}^p \overline f_j \partial_Jf_j$. On the other hand, since $$\left(\ds\sum_{j=1}^p i\partial_J f_j\wedge\overline{\partial_J f_j}\right)^{p-1}=(p-1)!\ds\sum_{j=1}^p \left(\ds\prod_{k\not=j}^pi\partial_J f_k\wedge\overline{\partial_J f_k}\right),$$ it follows that
$$
\begin{array}{lcl}
(|f|^2+\varepsilon^2)^{p+1}w_2(f,\varepsilon)&=&p!\ds\sum_{j=1}^p\left( \ds\prod_{k\not=j}^pi\partial_J f_k\wedge\overline{\partial_J f_k}\right)\wedge i\gamma\wedge\ov\gamma+{\cal O}\left(|f|\right)\wedge i\gamma\wedge\ov\gamma+{\cal O}\left(|f|^3\right)\\&=&p!|f|^2i\partial_J f_1\wedge\overline{\partial_J f_1}\wedge\cdots\wedge i\partial_J f_p\wedge\overline{\partial_J f_p}+{\cal O}\left(|f|^3\right).
\end{array}
$$
Consequently,
\begin{equation}\label{eq2.3}
w_2(f,\varepsilon)=p!{{|f|^2i\partial_J f_1\wedge\overline{\partial_J f_1}\wedge\cdots\wedge i\partial_J f_p\wedge\overline{\partial_J f_p}}\over{(|f|^2+\varepsilon^2)^{p+1}}}+{{{\cal O}\left(|f|^3\right)}\over{(|f|^2+\varepsilon^2)^{p+1}}}.
\end{equation}
Finally, regarding (\ref{eq2.1}),(\ref{eq2.2}) and (\ref{eq2.3}), we obtain
\begin{equation}\label{eq2.4}
\left(i\partial_J\ov\partial_J\log(|f|^2+\varepsilon^2)\right)^p=p!{{\varepsilon^2i\partial_J f_1\wedge\overline{\partial_J f_1}\wedge\cdots\wedge i\partial_J f_p\wedge\overline{\partial_J f_p}}\over{(|f|^2+\varepsilon^2)^{p+1}}}+{{{\cal O}\left(|f|^3\right)+\varepsilon^2{\cal O}\left(|f|\right)}\over{(|f|^2+\varepsilon^2)^{p+1}}}.
\end{equation}
Now let $z=(f_1,\cdots,f_p,z_{p+1},\cdots,z_n)$ be a local coordinates and we put $\xi_j^\star=\partial_J f_j$ for $j=1,\cdots,p$. Also, we denote by $\xi_j=\left({\partial\over{\partial z_j}}\right)^{1,0}$ and by $\xi_j^\star$ the associated dual vector, for every $j\in I_0=\{p+1,\cdots,n\}$. It follows that $\left(\xi_j^\star\right)_{1\leq j\leq n}$ is a smooth local frame of the bundle $T_{1,0}^\star\Omega$. With respect to this frame, let $\psi(z)=\sum_{|I|=|J|=n-p}\psi_{IJ}(z)\xi_I^\star\wedge \ov\xi_J^\star$ be a test form, and let us denote $d\lambda=\ds\prod_{j=1}^n{i\over 2}\xi_j^\star\wedge\overline{\xi_j^\star}$. Choose a constant $\epsilon(I_0)$ such that $$\epsilon(I_0) i\partial_J f_1\wedge\overline{\partial_J f_1}\wedge\cdots\wedge i\partial_J f_p\wedge\overline{\partial_J f_p}\wedge \xi_{I_0}^\star\wedge\overline{\xi_{I_0}^\star}=d\lambda.$$
Then  using (\ref{eq2.4}), we see that
$$
\left\langle \left(i\partial_J\ov\partial_J\log(|f|^2+\varepsilon^2)\right)^p,\psi\right\rangle ={{p!}\over{\epsilon(I_0)}}\ds\int{{\varepsilon^2 \psi_{I_0I_0}(z)}\over{(|f|^2+\varepsilon^2)^{p+1}}}d\lambda+\ds\int{{\left[{\cal O}\left(|f|^3\right)+\varepsilon^2{\cal O}\left(|f|\right)\right]\wedge\psi}\over{(|f|^2+\varepsilon^2)^{p+1}}}.$$
If we put $f_j=\varepsilon w_j$ in the first integral, we get
$$\begin{array}{lcl}
\ds\left\langle \left(i\partial_J\ov\partial_J\log(|f|^2+\varepsilon^2)\right)^p,\psi\right\rangle &=&\ds {{p!}\over{\epsilon(I_0)}}\ds\int{{\psi_{I_0I_0}(\varepsilon w_1,\cdots,\varepsilon w_p,z_{p+1},\cdots,z_n)}\over{(|w|^2+1)^{p+1}}}\widetilde{d\lambda}\\
&&\hfill+\ds\int{{\left[{\cal O}\left(|f|^3\right)+\varepsilon^2{\cal O}\left(|f|\right)\right]\wedge\psi}\over{(|f|^2+\varepsilon^2)^{p+1}}},
  \end{array}
$$
where $\widetilde{d\lambda}$ is a modification of $d\lambda$ by replacing $\partial_J f_j$ by $\varepsilon\partial_J w_j$ for $j=1,\cdots,p$. Letting $\varepsilon\l 0$ and using the fact that $$\int_{\bC^p}{dw\over{(|w|^2+1)^{p+1}}}={{\pi^p}\over{p!}},$$ it is not hard to see that
$$
\begin{array}{lcl}
\ds\left\langle \left({i\over\pi}\partial_J\ov\partial_J\log |f|^2\right)^p,\psi\right\rangle&=&\ds\int\psi_{I_0I_0}(0,\cdots,0,z_{p+1},\cdots,z_n)d\lambda+\ds\int{{\left[{\cal O}\left(|f|^3\right)\right]\wedge\psi}\over{|f|^{2p+2}}}\\&=&\ds\int [Z]\wedge\psi+\ds\int{{\left[{\cal O}\left(|f|^3\right)\right]\wedge\psi}\over{|f|^{2p+2}}}\\&=&\left\langle [Z],\psi\right\rangle+\left\langle R_J(f),\psi\right\rangle,
\end{array}
$$
where $[Z]$ is the current of integration on $Z$ and $R_J(f)$ is a current with coefficients ${\cal O}\left(|f|^{1-2p}\right)$. Before finishing the proof we point out that if $J$ is integrable, then $\ov\partial f_j\equiv 0$, for $j=1,\cdots,p$. Hence, the term ${\cal O}\left(|f|\right)$ in (\ref{eq2.2}) and ${\cal O}\left(|f|^3\right)$ in (\ref{eq2.3}) turn out to be zero. This implies in particular that the extra current $R_J(f)$ vanishes.
 \end{proof}
\section{Restriction of closed positive currents on J-analytic subsets}
\subsection{J-analytic subsets}
In this subsection we are going to introduce the notion of a {\it $J$-analytic} subset in an almost complex manifold $(M,J)$. Such subsets should be considered as almost complex analogues of ``classical'' analytic subsets. According to \cite{5}, a subset $A$ of $M$ is said to be {\it $J$-analytic} in $(M,J)$ if $A$ admits a stratification $A=A_s\cup A_{s-1}\cup \cdots\cup A_0$, where $A_0$ is an almost complex submanifold in $M$ and for $j=1,\cdots,s$, $A_j$ is a closed almost complex submanifold of $M\smallsetminus\ds\cup_{k=0}^{j-1}A_k$. It is important to point out that this definition is far from being ideal as shown by the following example :
$$A=\left\{(z,w)\in\bC^2;\ w\not=0,\ z=e^{1/w}\right\}\cup \left\{w=0\right\}.$$
It is not hard to see that $A$ is not analytic (in the complex sense) but it is $J_{st}$-analytic in the sense of \cite{5} where $J_{st}$ is the standard complex structure on $\bC^2$. This example leads us to formulate another more appropriate definition of a {\it $J$-analytic} subset that has at least the minimum requirement of being compatible with the concept of {\it $J_{st}$-analytic} subset in the integrable situation. For this aim, we introduce the following definition :
 \begin{defn} \label{def2}
    We say that $A$ is a {\it $J$-analytic subset} of $(M,J)$ of dimension $p$ if there exist a finite sequence of closed subsets
 $$\emptyset=A_{-1}\subset A_0\subset A_1\subset\cdots\subset A_p=A ,$$
where $A_j\smallsetminus A_{j-1}$ is a smooth almost complex submanifold in $M\smallsetminus A_{j-1}$, of complex dimension $j$ and has a locally finite $2j$-Hausdorff dimension in the neighborhood of every point of $M$. We say that $A$ is of {\it pure} complex dimension $p$ if moreover we have $A_{j-1}\subset\overline{A_j\smallsetminus A_{j-1}}$, for $j=0,...,p$. If the $p$-dimensional strata $A_p\smallsetminus A_{p-1}$ is connected we say that $A$ is irreducible.
\end{defn}
It is clear that $A_0$ is a smooth almost complex submanifold in $M$. Moreover, it should be mentioned that the inclusions in the above definition are not necessarily strict (the $j$-dimensional strata $A_j\smallsetminus A_{j-1}$ may be empty).
\begin{rem}\
    \begin{enumerate}
        \item  Notice that the previous definition for the almost complex setting does coincide with usual analytic subsets in the integrable case (this follows from the standard extension theorem for analytic sets with locally bounded area). Notice that our above example was constructed precisely in such a way that the area of $z=e^{1/w}$ is not locally finite near $w=0$.
        \item In order to justify the above definition let us recall that every closed $J$-complex curve $A$ of $(M,J)$ is $J$-analytic. Indeed, we write $\emptyset=A_{-1}\subset A_0\subset A_1=A ,$ where $A_0$ is the singular part of $A$ which is discrete. More generally, every almost complex submanifold is a $J$-analytic subset.
        \item According to the terminology introduced in \cite{2}, a regular complete $J$-pluripolar (r.c.p for short) subset $A$ of $(M,J)$ is the $-\infty$ locus of a $J$-psh function, which is of class ${\mathscr C}^2$ away from $A$. In the same paper it was proved that every almost complex submanifold of $(M,J)$ is locally r.c.p. According to definition \ref{def2}, this enables us to deduce without difficulty that every $J$-analytic subset $A$ is a locally regular complete $J$-pluripolar subset away from the singular part of $A$. Obviously, a natural question arises here : does every $J$-analytic subset is a (locally) regular complete pluripolar set? This is a well-known result when the structure $J$ is integrable.
    \end{enumerate}
 \end{rem}
\subsection{Closed positive currents on J-analytic subsets}
Our next result concerns the restriction of closed positive currents on $J$-analytic subsets. First, recall that in terms of currents, if $A$ is a $J$-analytic subset of complex dimension $p$ then $[A]$ defines a $(p,p)$-closed positive current by integrating $(p,p)$-test forms on the components of $A$ of dimension $2p$. More precisely, assume that $\emptyset=A_{-1}\subset A_0\subset A_1\subset\cdots\subset A_p=A$ is a sequence as in definition \ref{def2} and let $Y=A_p\smallsetminus A_{p-1}$. Since $Y$ is a smooth almost complex submanifold in $M\smallsetminus A_{p-1}$, then the integration on $Y$ defines a positive closed current on $M\smallsetminus A_{p-1}$. Moreover, $Y$ has a locally finite $2p$-Hausdorff dimension in the neighborhood of every point of $M$, hence $[Y]$ has locally finite mass across $A_{p-1}$ and therefore $[A]=\widetilde{[Y]}$ is the trivial extension of $[Y]$. L'et us recall the following result :
\begin{lem} Assume that $T$ is a positive closed (resp.plurisubharmonic) current on $(M,J)$ and $A$ is a $J$-analytic subset of complex dimension $p$, then the cut-off ${\1}_A T$ is also a positive and closed (resp.plurisubharmonic) current supported by $A$.
\end{lem}
\begin{proof} The problem is local, so we may assume that $M$ is an open subset $\Omega$ of $\bC^{n}$. Lat $\emptyset=A_{-1}\subset A_0\subset A_1\subset\cdots\subset A_p=A$ be a sequence of $(\Omega,J)$ as in definition \ref{def2} and let $T$ be a positive closed (the case when $T$ is positive and plurisubharmonic is similar) of $(\Omega,J)$. Thank's to a result of \cite{2}, every almost complex submanifold of $(\Omega,J)$ is locally r.c.p and the cut-off of $T$ by a r.c.p subset is also positive and closed. This is the case for ${\1}_{A_0} T$ because $A_0$ is a smooth almost complex submanifold in $\Omega$. Hence, by using an induction on the dimension of $A$ and observe that ${\1}_{A_j}T={\1}_{{A_j}\smallsetminus A_{j-1}}T+{\1}_{A_{j-1}}T$, for $j\leq p$, we need only to prove that $R={\1}_{{A_j}\smallsetminus A_{j-1}}T$ is positive and closed in $\Omega$. Since $A_j\smallsetminus A_{j-1}$ is a smooth almost complex submanifold in $\Omega\smallsetminus A_{j-1}$, again by \cite{2} the current $R$ is positive and closed on $\Omega\smallsetminus A_{j-1}$ (notice that $R$ don't carry any mass on $A_{j-1}$). Let's write $\Omega\smallsetminus A_{j-1}=\left(\Omega\smallsetminus A_{j-2}\right)\setminus \left(A_{j-1}\smallsetminus A_{j-2}\right)$. Since $A_{j-1}\smallsetminus A_{j-2}$ is a smooth almost complex submanifold in $\Omega\smallsetminus A_{j-2}$ and $R$ has locally finite mass near $A_{j-1}\smallsetminus A_{j-2}$ then the trivial extension $\widetilde R$ of $R$ is also positive and closed on $\Omega\smallsetminus A_{j-2}$ (see \cite{2}). As the current $R$ is supported by ${A_j}\smallsetminus A_{j-1}$, we can deduce that $R=\widetilde R$. The proof was completed by repeating the above argument and by using the induction hypothesis.
\end{proof}
In the case when $J$ is integrable, this lemma was proved by El Mir in the more general setting when $A$ is a pluripolar subset. Moreover, lemma 1 extends a result due to [2], if $A$ is an almost complex submanifold. Notice also that by the same idea of the proof of lemma 1, we can easily see that the current of integration $[A]$ on a $J$-analytic subset is positive and closed. In the same direction concerning the restriction of currents, assume that $T$ is a positive plurisubharmonic current of bibimension $(p,p)$ supported by a subset $A$ of vanishing $2p$-Hausdorff dimension in an almost complex manifold $(M,J)$. By a line of arguments that goes back to Siu \cite[lemma 3.3]{8} and by using the fact that the Lelong function (that induces the existence of the Lelong number) is ``almost'' increasing (see \cite{3}), it is not hard to prove that ${\1}_{A}T=0$.  When $A$ is $J$-analytic of dimension $p$, we obtain:
    \begin{theo}
        Let $T$ be a closed positive current of bibimension $(p,p)$ on an almost complex manifold $(M,J)$. Let $A$ be a $J$-analytic subset of $(M,J)$ of dimension $p$. Then, we have $${\1}_A T=m_A[A] ,$$
        where $m_A=\ds\inf_{x\in A}\nu_T(x)$ is the generic Lelong number of $T$ along $A$.
    \end{theo}
\begin{proof} Assume that $\emptyset=A_{-1}\subset A_0\subset A_1\subset\cdots\subset A_p=A$ is a sequence as in definition \ref{def2}. Since positiveness of currents is a local problem, it is enough to work in a neighborhood of each point $a\in A$. Following Demailly, let $(z_1,\cdots,z_n)$ be an almost complex coordinates around $a$ such that $\ov\partial z_j=O(|z|)$. Denote by $\beta=id\overline\partial|z|^2$, $\beta_1={i\over 2}\sum_{j=1}^n\partial z_j\wedge\overline{\partial z_j}$ and $\sigma_T(B(a,r))=\int_{B(a,r)}T\wedge\beta_1^p$. It is not hard to see that $\beta=\beta_1+O(|z|)$ and $\beta^p=\beta_1^p+O(|z|)$, (see \cite{2}).
\vskip0.1cm
{\bf Claim.} {\it For every $\varepsilon>0$ there exist $r_0>0$ and $c>0$ such that
 $$\Theta_{r,\varepsilon}={\1}_A T-m_A\left(1-(2pc+\varepsilon)r\right)[A]\geq 0,\quad\forall 0<r<r_0.$$}%
{\it Proof of the claim.} Consider the following important fact : the map $r\mapsto \sigma_T(B(a,r)){{(1+rc)^{2p}p!}\over{r^{2p}\pi^p}}$ is increasing for some constant $c>0$. It follows that
$$\sigma_T(B(a,r))\geq m_A{{\pi^p}\over{p!}}{{r^{2p}}\over{(1+rc)^{2p}}}=m_A{{\pi^p}\over{p!}}r^{2p}\left(1-2pcr+o(r)\right).$$
Given $\varepsilon>0$, there exists $r_0>0$ such that for every $0<r<r_0$ we have
$$\sigma_T(B(a,r))\geq m_A{{\pi^p}\over{p!}}r^{2p}\left(1-(2cp+\varepsilon)r\right)=\sigma_{[A]}(B(a,r))m_A\left(1-(2cp+\varepsilon)r\right).$$
It follows that $\int {{\1}_{B(a,r)}}\Theta_{r,\varepsilon}\wedge\beta_1^p\geq 0$, where $\Theta_{r,\varepsilon}=T-m_A\left(1-(2cp+\varepsilon)r\right)[A]$. Now, let $f$ be a positive continuous function with compact support in $M$ and let $$g_\delta(z)=\ds\int_A {\1}_{B(a,\delta)}(z)f(a)d\lambda(a)=\ds\int_{a\in A\cap{B(z,\delta)}}f(a)d\lambda(a).$$
Then, we have $\int g_\delta\Theta_{r,\varepsilon}\wedge\beta_1^p\geq 0$. Moreover, it is not hard to see that the function $g_\delta$ is continuous, and that when $\delta$ tends to $0$, the function ${{p!}\over{\pi^p\delta^{2p}}}g_\delta$ converges to $f$ on $A$ and to $0$ on $M\smallsetminus A$. Thus we infer the inequality $\int {\1}_A f\Theta_{r,\varepsilon}\wedge\beta_1^p\geq 0$. Since $f$ is arbitrary, we see that the measure ${\1}_A\Theta_{r,\varepsilon}\wedge\beta_1^p$ is positive. On the other hand, thanks to a result of \cite{2}, in a neighborhood $U$ of $a$ we have
$$A\cap U=\left\{x\in U,\ f_{p+1}(x)=\cdots=f_n(x)=0,\ \ov\partial_Jf_j=0\
{\rm on}\ A\cap U\ {\rm and}\
\partial_Jf_{p+1}\wedge\cdots\wedge\partial_Jf_n\not=0\ {\rm on}\ U \right\}.$$
The current ${\1}_A\Theta_{r,\varepsilon}$ is of order zero, supported by $A$ and  $d$-closed. Therefore, in virtue of the classical support theorem of Federer, for every $p+1\leq j\leq n$, we obtain
\begin{equation}\label{eq3.1}
    {\1}_A\Theta_{r,\varepsilon}\wedge\partial f_j={\1}_A\Theta_{r,\varepsilon}\wedge\overline{\partial f_j}=0.
\end{equation}
Consider now the almost complex system $\left(z_1,\cdots,z_p,z_{p+1}=f_{p+1},\cdots,z_n=f_n\right)$ such that $(\partial z_j)_j$ is a local basis of $T_{1,0}^\star M_{|U}$. For $k=1,\cdots,p$, let $\alpha_k=\sum_j\alpha_{kj}\partial z_j\in{\cal D}_{1,0}(U)$. Then a simple computation gives
$$i\alpha_1\wedge\overline\alpha_1\wedge\cdots\wedge i\alpha_p\wedge\overline\alpha_p=|{\rm det}(\alpha_{kj})|_{1\leq j,k\leq p}^2 i\partial z_1\wedge\overline{\partial z_1}\wedge\cdots\wedge i\partial z_p\wedge\overline{\partial z_p}+\gamma,$$
where $\gamma$ is a form which contains at least one $\partial f_j$ or $\overline{\partial f_j}$ for some $j$. In view of (\ref{eq3.1}) it follows that
$$
\begin{array}{lcl}
{\1}_A\Theta_{r,\varepsilon}\wedge i\alpha_1\wedge\overline\alpha_1\wedge\cdots\wedge i\alpha_p\wedge\overline\alpha_p&=&|{\rm det}(\alpha_{kj})|^2{{\1}_A}\Theta_{r,\varepsilon}\wedge i\partial z_j\wedge\overline{\partial z_j}\wedge\cdots\wedge i\partial z_p\wedge\overline{\partial z_p}\\&=&|{\rm det}(\alpha_{kj})|^2{\1}_A\Theta_r\wedge\beta_1^p\geq 0.
\end{array}
$$
This means that the current ${{\1}_A}\Theta_{r,\varepsilon}$ is positive, and the proof of the claim is complete. In order to finish the proof of the theorem, we remark that ${\1}_A T\geq m_A [A]$, by letting $r\l 0$. On the other hand, since $T$ is positive and closed, we see by the Federer support theorem that ${\1}_A T=c[A]$, for some $c\geq 0$. Since $T\geq {\1}_A T=c[A]$, we have $\nu_T(x)\geq c,\ \forall x\in A$ and therefore $m_A\geq c$. The opposite inequality is clear because ${\1}_A T\geq m_A [A]$.
\end{proof}
Before terminating this section, we state two interesting related problems pertaining to the notion of $J$-analytic subsets :\vskip0.1cm
{\bf P1.} {\it What can be said about loci of singularities of a $J$-analytic subset?}\vskip0.1cm
{\bf P2.} {\it Let $T$ be a closed positive current of bidimension $(p,p)$ on an almost complex manifold

\hskip0.8cm $(M,J)$. Fix $c>0$. Is the set $E_c=\{x\in M,\ \nu_T(x)\geq c\}$ a $J$-analytic subset of $M$?}\vskip0.1cm
It is at least well known that $E_c$ has a locally finite $2p$-dimensional
Hausdorff measure.
\section{Independence of  Lelong numbers from the coordinates}
Our primary goal in this section is to give a sufficient condition guaranteeing that the Lelong number of a negative plurisubharmonic current $T$ defined on an almost complex manifold exists. Following \cite{3}, the same result holds when $T$ is positive and psh without requiring any further condition. The second aim is to prove that the Lelong numbers of a positive (or negative) psh current are independent on the coordinate systems. In this way, we give a positive answer to a question stated in \cite{3}. As these problems are local, assume that $0\in M$, where $(M,J)$ is an almost complex manifold and let $(z_1,\cdots,z_n)$ be a coordinate system at $0$ such that $\ov\partial z_j=O(|z|)$. We keep here the notations used in \cite{2} :
\begin{eqnarray*}
\beta_1={i\over 2}\partial_J \overline\partial_J |z|^2,\qquad&&\beta={i\over 2}d\overline\partial|z|^2=\beta_1-{i\over
2}\theta_J \overline\partial_J |z|^2 + {i\over 2}\overline\partial_J^2|z|^2,\\
\sigma_T(B(r))=\int_{B(r)}T\wedge\beta_1^p,\,\quad&&\nu_T(r)={\sigma_T(r)\over\tau_pr^{2p}},
\end{eqnarray*}
where $T$ is a positive current of bidimension $(p,p)$ on $(M,J)$, $\tau_p$ is the volume of the unit ball of $\bC^p$ and $B(r)=\{z,\ |z|<r\}$. According to these notations, we state :
\begin{prop} Let $T$ be a negative plurisubharmonic current of bidimension $(p,p)$ on $(M,J)$. Assume that $t\mapsto {{\nu_{i\partial_J \ov\partial_J T}(t)}\over t}$ is locally integrable in a neighborhood of zero. Then, the Lelong number of $T$ at zero exists and is equals to the limit of $\nu_T(r)$, when $r$ tends to zero.
\end{prop}
 In the case when $J$ is integrable, this reduces to a result of \cite{4}. Moreover, it should be mentioned that in the same paper, it was shown that the above integrability condition is a sufficient condition that is not necessary.
\begin{proof} Without loss of generality, we may assume that $T$ is positive and $i\partial_J \ov\partial_J T$ is negative. Let
$$\ov\sigma_T(r)=\int_{B(r)}\left(T\w\beta^p- {i\over
2}T\w\ov\partial_J |z|^2\w\partial_J\beta_{p-1}+ {i\over
2}\ov\theta_J T\w\partial_J |z|^2\w\beta^{p-1}\right)\quad{\rm and}\quad \ov\nu_T(r)={\ov\sigma_T(r)\over r^{2p}}.$$
Taking into account the estimations of the forms $\beta^p,\ \partial_J |z|^2\w\beta^{p-1}$ and $\partial_J |z|^2\w\beta^{p-1}$ already given by \cite{2}, it is not hard to see that $\overline\nu_T(r)-\nu_T(r)=O(r)$. Let us consider the map :
$$
\begin{array}{lcl}
f(r)&=&\ov\nu_T(r)+ {1\over
r^{2p}}\ds\int_{0}^{r}tdt\int_{B(t)}i\partial_J \ov\partial_J
T\w\beta^{p-1}-\ds\int_{0}^{r}{{dt}\over
t^{2p-1}}\int_{B(t)}i\partial_J \ov\partial_J
T\w\beta^{p-1}
\\&=&\ov\nu_T(r)+\ds\int_0^r\left({{t^{2p}}\over
{r^{2p}}}-1\right){{\nu_{i\partial_J \ov\partial_J
T}(t)}\over t}dt=\ov\nu_T(r)+g(r) .
\end{array}
$$
Observe that $g$ is positive and increasing. Then, there exists a constant $\delta>0$ such that for every $1\gg r_2>r_1>0$, we have
$$
\begin{array}{lcl}
    f(r_2)-f(r_1)&=&\ov\nu_T(r_2)-\ov\nu_T(r_1)+ g(r_2)-g(r_1)\\
    &\geq&\ds\ov\nu_T(r_2)-\ov\nu_T(r_1)\\&\geq& {-\delta\over {r_1^{2p-1}}}\left(\ov\sigma_T(r_2)-\ov\sigma_T(r_1)\right)-
\delta\left({1\over {r_1}^{2p}}-{1\over {r_2}^{2p}}\right)r_2\ov\sigma_T(r_2).
\end{array}
$$
It should be mentioned that the last inequality was proved by \cite{3} in the case when $T$ is positive and psh and a simple computation shows that the same inequality remains valid in our setting. Thus, for $r_1=r,$ $r_2=r+h$ and letting $h\l 0$, we get
$$f'(r)=\ov\nu_T'(r)+g'(r)\geq {-\delta\over
{r^{2p-1}}}\left(r^{2p}\ov\nu_T'(r)+2pr^{2p-1}\ov\nu_T(r)\right)-2p\delta
\ov\nu_T(r)=-\delta r\ov\nu_T'(r)-4p\delta\ov\nu_T(r),$$
in other word we have $(1+\delta r)\ov\nu_T'(r)+4p\delta\ov\nu_T(r) + g'(r)\geq 0$. However, by multiplying this inequality with the positive term $(1+\delta r)^{4p-1}$, we obtain $$\left[(1+\delta r)^{4p}\ov\nu_T(r)\right]'+(1+\delta r)^{4p-1}g'(r)\geq 0.$$ Hence, for $r\ll 1$, it follows that
$$\left[(1+\delta r)^{4p}\ov\nu_T(r) + (1+\delta r)^{4p-1}g(r)\right]'\geq(4p-1)(1+\delta r)^{4p-2}\delta g(r)\geq 0.$$
In particular, the quantity $(1+\delta r)^{4p}\ov\nu_T(r)+(1+\delta r)^{4p-1}g(r)$ admits a limit when $r\l 0$. On the other hand the integrability condition in the theorem implies $\ds\lim_{r\l 0}g(r)=0$, which completes the proof.
\end{proof}
 We are now going to prove the main result in this section. Namely, the Lelong numbers of a positive (or negative) psh current are independent of coordinate systems. More precisely, we prove:
\begin{theo} Let $T$ be a psh current of bidimension $(p,p)$ on $(M,J)$. Suppose that either $T$ is positive, or $T$ is negative and the map $t\mapsto {{\nu_{i\partial_J \ov\partial_J T}(t)}\over t}$ is locally integrable in a neighborhood of zero. Then the Lelong number of $T$ does not depend on the coordinate system.
\end{theo}
\begin{proof} Since the problem is local, let $z^1=(z_1^1,\cdots,z_n^1)$ and $z^2=(z_1^2,\cdots,z_n^2)$ be two systems of coordinates at $0$. Denote by $\varphi=|z^1|^2,$ $\psi=|z^2|^2$ and for every $\varepsilon\ll 1$, $l>1$, we put $\psi_{\varepsilon,l}=\psi^l+\varepsilon\varphi$. Then, since $\ds\lim_{\varphi\l 0}{{\log\psi}\over{\log\varphi}}=1$, it is clear that $\ds\lim_{\varphi\l 0}{\psi^l\over\varphi}=0$, this means that $\psi_{\varepsilon,l}\approx\varepsilon\varphi,$ when $\varphi\l 0$. A direct computation gives (see \cite{3})
$$id\ov\partial\psi_{\varepsilon,1}=id\ov\partial(\psi +\varepsilon\varphi)=id\ov\partial\psi+id\ov\partial(\varepsilon\varphi)=\varepsilon\beta_1^1+\beta_1^2+ O(\varepsilon |z^1|)+O(|z^2|),$$
where $\beta_1^1={i\over 2}\sum_{j=1}^n\partial_J  z_{j}^1\w
\overline{\partial_J  z^1_j}$ and $\beta_1^2={i\over 2}\sum_{j=1}^n\partial_J  z_{j}^2\w
\overline{\partial_J  z^2_j}$. On the other hand
$$id\ov\partial\psi^l=i\partial\ov\partial\psi^l+i\ov\partial^2\psi^l-i\theta\ov\partial\psi^l-i\ov\theta\ov\partial(\psi^l) .$$ Using the fact that $\ov\partial\psi=O(|z^2|^2)$ and $i\partial\ov\partial\psi=\beta_1^2+O(|z^2|)$, we get
$$i\partial\ov\partial\psi^l=l(l-1)\psi^{l-2}i\partial\psi\wedge\ov\partial\psi+l\psi^{l-1}i\partial\ov\partial\psi=O\left(|z^2|^{2l-2}\right)\beta_1^2+O(|z^2|).$$
It is easy to see that the forms $\ov\partial^2\psi,\ \theta\ov\partial\psi$ and $\ov\theta\ov\partial\psi$ have $O(|z^2|)$ as coefficients. Hence we deduce
$$
i\ov\partial^2\psi^l=l\psi^{l-1}i\ov\partial^2\psi^l=O(|z^2|),\qquad
i\theta\ov\partial\psi^l=l\psi^{l-1}i\theta\ov\partial\psi=O(|z^2|),$$
$$i\ov\theta\ov\partial\psi^l=l\psi^{l-1}i\ov\theta\ov\partial\psi=O(|z^2|) .$$
It follows that
$$id\ov\partial\psi^l=O\left(|z^2|^{2l-2}\right)\beta_1^2+O(|z^2|) .$$
Finally, we obtain
\begin{equation}\label{eq4.1}
id\ov\partial\psi_{\varepsilon,l}=O\left(|z^2|^{2l-2}\right)\beta_1^2+\varepsilon\beta_1^1+O(|z^2|)+O(\varepsilon|z^1|) .
\end{equation}
\noindent
$\bullet$ Suppose that $T$ is positive and psh. By positiveness and by using (\ref{eq4.1}), we see that
$$
\ds{1\over{r^{2p}}}\ds\int_{\{\psi_{\varepsilon,l}<r^2\}}T\wedge (id\ov\partial\psi_{\varepsilon,l})^p\geq\ds{1\over{r^{2p}}}\ds\int_{\{\psi_\varepsilon<r^2\}}T\wedge (\varepsilon\beta_1^1)^p+O(r).$$
On the other hand for every $\delta>0$, there exists $\eta>0$ such that if $\varphi<\eta$ then $\psi_{\varepsilon,l}\leq\varepsilon(1 +\delta)\varphi$. Then, one has the inclusion $\{\varepsilon\varphi<r^2(1+\delta)^{-1}\}\subset\{\psi_{\varepsilon,l}<r^2\}$. Consequently
\begin{eqnarray*}
{1\over{r^{2p}}}\ds\int_{\{\psi_{\varepsilon,l}<r^2\}}T\wedge (\varepsilon\beta_1^1)^p&\geq& {1\over{r^{2p}}}\ds\int_{\{\varepsilon\varphi<r^2(1+\delta)^{-1}\}}T\wedge (\varepsilon\beta_1^1)^p\\
&=&(1+\delta)^{-p}\left(r(\varepsilon(1+\delta))^{-1/2}\right)^{-2p}\ds\int_{\left\{|z^1|<r(\varepsilon(1+\delta))^{-1/2}\right\}}T\wedge (\beta_1^1)^p\\
&=&(1+\delta)^{-p}\nu_T\left(\varphi,r(\varepsilon(1+\delta))^{-1/2}\right),
\end{eqnarray*}
where $\nu_T(\varphi,t)={1\over{t^{2p}}}\int_{\{|z^1|<t\}}T\wedge (\beta_1^1)^p$. Observe that
$id\ov\partial\psi_{\varepsilon,l}\l id\ov\partial\psi_{\varepsilon,1}$, when $l\l 1$. Then by passing to the limit $l\l 1,\delta\l0$ we obtain
\begin{equation}\label{eq4.2}
\ds{1\over{r^{2p}}}\ds\int_{\{\psi_{\varepsilon,1}<r^2\}}T\wedge (id\ov\partial\psi_{\varepsilon,1})^p\geq\nu_T\left(\varphi,r/\sqrt{\varepsilon}\right)+O(r).
\end{equation}
On the other hand, it was shown in [3] that the quantity $\overline\nu_T(\varphi,t)$ converges to $\nu_T(\varphi)$ when $t\l 0$. More precisely the map $t\mapsto(1+ct)^{4p}\overline\nu_T(\varphi,t)$ is increases near $0$ for some constant $c>0$. Moreover, it was noticed in the begining of the proof of prop.1 that $\overline\nu_T(\varphi,t)$ differs from $\nu_T(\varphi,t)$ by a $O(t)$ term for $t$ sufficiently small. Therefore from (\ref{eq4.2}) we get the following
\begin{equation}\label{eq4.3}
\ds{1\over{r^{2p}}}\ds\int_{\{\psi_{\varepsilon,1}<r^2\}}T\wedge (id\ov\partial\psi_{\varepsilon,1})^p\geq\nu_T(\varphi)+O\left(r/\sqrt{\varepsilon}\right)+O(r).
\end{equation}
Let $r_\varepsilon=\varepsilon^{3/2}$. Thank's to the expression of $id\ov\partial\psi_{\varepsilon,1}$ and observe that $\{\psi_{\varepsilon,1}<r^2\}\subset \{\psi<r^2\}$, we derive that the integral in (\ref{eq4.2}) is bounded above by $\nu_T(\psi,r_\varepsilon)+O(\varepsilon)$. It follows that $\nu_T(\psi,r_\varepsilon)+O(\varepsilon)\geq\nu_T(\varphi)+O(\varepsilon)$ and hence we let $\varepsilon\l 0$ to obtain the inequality $\nu_T(\psi)\geq\nu_T(\varphi)$. Finally, observe that the equality is obtained by reversing the roles of  $\psi$ and $\varphi$.\\
$\bullet$ Suppose that $T$ is negative psh and that $t\mapsto {{\nu_{i\partial_J \ov\partial_J T}(t)}\over t}$ is locally integrable in a neighborhood of zero. Without loss of generality, we may assume that $-T$ is negative psh. By repeating the same argument of the previous case of positive psh currents, we see that the inequality (\ref{eq4.2}) remains valid. Moreover, by turning back to the proof of prop.1 ther exists a constant $c>0$ such that the map $t\mapsto(1+ct)^{4p}\overline\nu_T(\varphi,t)+(1+ct)^{4p-1}g(t)$ is increases near $0$. Hence, it is not hard to see that the inequality (\ref{eq4.3}) becomes
$$
\ds{1\over{r^{2p}}}\ds\int_{\{\psi_{\varepsilon,1}<r^2\}}T\wedge (id\ov\partial\psi_{\varepsilon,1})^p\geq\nu_T(\varphi)+O\left(r/\sqrt{\varepsilon}\right)-\frac{g(r/\sqrt{\varepsilon})}{\left(1+c(r/\sqrt{\varepsilon})\right)}+O(r).
$$
Thus we complete the proof by following the same argument after inequality (\ref{eq4.3}) and by using the fact that $\ds\lim_{r\l 0}g(r)=0$.
\end{proof}
\section*{Acknowledgments} The author would like to thank Professor Jean-Pierre Demailly for many fruitful discussions concerning this article.

\end{document}